\documentclass[11pt,oneside]{article}
\usepackage{amsmath,amsfonts,amssymb,latexsym,amsthm}
\usepackage[dvips]{graphics}

\newtheorem{theorem}{Theorem}[section]
\newtheorem{lemma}{Lemma}[section]

\newtheorem{corollary}{Corollary}[section]

\theoremstyle{remark}

\numberwithin{equation}{section}

\begin{document}
\author{L. Ciobanu, B. Fine, G. Rosenberger}
\date{}
\title{The Surface Group Conjecture:\\ Cyclically Pinched and Conjugacy Pinched One-Relator Groups}

\maketitle

\begin{abstract}
The general {\bf surface group conjecture} asks whether a one-relator group where every subgroup of finite index is again one-relator and every subgroup of infinite index is free (property IF) is a surface group. We resolve several related conjectures given in [FKMRR]. First we obtain the Surface Group Conjecture B for cyclically pinched and conjugacy pinched one-relator groups. That is: if $G$ is a cyclically pinched one-relator group or conjugacy pinched one-relator group satisfying property IF then $G$ is free, a surface group or a solvable Baumslag-Solitar Group. Further combining results in [FKMRR] on Property IF with a theorem of H. Wilton [W] and results of Stallings [St] and Kharlampovich and Myasnikov [KhM4]  we show that Surface Group Conjecture C proposed in [FKMRR] is true, namely: If $G$ is a finitely generated nonfree freely indecomposable fully residually free group with property IF, then $G$ is a surface group.
\medskip


\end{abstract}

\section{Introduction} The \textbf {surface group conjecture} as originally proposed in the Kourovka notebook by Melnikov 
was the following problem. 

\textbf{Surface Group Conjecture.} \textit{Suppose that $G$ is a residually finite non-free, non-cyclic
one-relator group such that every subgroup of finite index is again a one-relator group. Then G is
a surface group. }

In this form the conjecture is false. Recall that the Baumslag-Solitar groups $BS(m,n)$ are the groups
$$BS(m,n) = <a,b; a^{-1}b^ma = b^n>.$$
If $|m| = |n|$ or either $|m| =1$ or $|n|= 1$ these groups are residually finite.  In all other cases these groups are  not residually finite. Further $BS(m,n)$ is Hopfian if and only if $m = \pm 1$ or $n = \pm 1$ or $m,n$ have the same set of primes  ([CoLe]). If either $|m| = 1$ or $|n| = 1$ every subgroup of finite index is again a Baumslag-Solitar group and therefore a one-relator group, and every subgroup of infinite index is infinite cyclic and therefore a free group of rank one ([GKhM]). It follows that besides the surface groups, the groups $BS(1,m)$, also satisfy Melnikov's question. We then have the following modified conjecture. 

\vspace{.3cm}

\textbf{Surface Group Conjecture A.} \textit{Suppose that $G$ is a residually finite non-free, non-cyclic
one-relator group such that every subgroup of finite index is again a one-relator group. Then $G$ is either
a surface group or a Baumslag-Solitar group $BS(1,m)$ for some integer $m$.} 

\vspace{.3cm}

We note that the groups $BS(1,1)$ and $BS(1,-1)$ are surface groups. In surface groups, subgroups of infinite index must be free groups.  To avoid the Baumslag-Solitar groups, $BS(1,m), |m| \ge 2$, Surface Group Conjecture A, was modified to:
 
 \vspace{.3cm}
 
 \textbf{Surface Group Conjecture B.} \textit{Suppose that $G$ is a non-free, non-cyclic one-relator
group such that every subgroup of finite index is again a one-relator group, there exists a noncyclic subgroup of infinite index and every subgroup of infinite index is a free group . Then $G$ is a surface group of genus $g \ge 2$.}

\vspace{.3cm}

In [FKMRR] the surface group conjecture was considered for fully residually free groups using the JSJ decomposition. A group $G$ has {\bf Property IF} if every subgroup of infinite index in $G$ is free. In [FKMRR] it was proved that a fully residually free group satisfying property IF is either a cyclically pinched one-relator group or a conjugacy pinched one-relator group. This led to the following:

\vspace{.3cm}

 \textbf {Surface Group Conjecture C.} \textit{Suppose that $G$ is a finitely generated nonfree freely
indecomposable fully residually free group with property IF. Then G is a surface group.}

\vspace{.3cm}

In Theorem \ref{mthm2} below, we settle Surface Group Conjecture C by combining results in [FKMRR] with a recent theorem of Wilton [W] and results of Kharlampovich and Myasnikov [KhM4] and Stallings [St]. This result also appears in Wilton [W] for one-ended limit groups.

\vspace{.3cm}

\textbf{Theorem (\ref{mthm2})} \textit{Suppose that $G$ is a finitely generated nonfree freely
indecomposable fully residually free group with property IF. Then G is a surface group. That is, Surface Group Conjecture C is true.}

\vspace{.3cm}

We then improve on a result of  Wilton (see Theorem 2.4)  by dropping the conditions of one-ended hyperbolic from the hypothesis.  Our main result concerns cyclically pinched and conjugacy pinched one-relator groups.

\vspace{.3cm}

\textbf{Theorem (\ref{mthm})} \textit{(1) Let G be a cyclically pinched one-relator group 
satisfying property IF.  Then $G$ is a free group or a surface group.}

\textit{$\hphantom{xx}$(2) Let G be a conjugacy pinched one-relator group 
satisfying property IF  Then $G$ is a free group,  a surface group or a solvable Baumslag-Solitar group.
}
\vspace{.3cm}

\section{Background Material and Necessary Results} Let $G$ be the fundamental group of a compact surface of
genus $g$.  Then $G$ has a one-relator presentation 
$$<a_1,b_1,...,a_g,b_g; [a_1,b_1]....[a_g,b_g]>$$
in the orientable case and 
$$<a_1,....a_g; a_1^2...a_g^2>$$
in the non-orientable case.   From covering space theory it follows that any subgroup of finite
index is again a surface group of the same or higher genus while any subgroup of infinite index must be a
free group. These results, although known since the early 1900's, were proved purely algebraically
using Reidemeister-Schreier rewriting by Hoare, Karrass and Solitar in 1971 [HKS 1,2]. It is well
known (see [FR]) that an orientable surface group can be faithfully represented as a discrete
subgroup of $PSL_2(\mathbb C)$ and hence each such group is linear.  It follows that surface groups
are residually finite. G.Baumslag [GB] showed that any orientable surface group of genus $\ge 2$ must be
residually free and 2-free from which it can be deduced using results of Remeslennikov [Re] and
Gaglione and Spellman [GS] that they are fully residually free (see section 2).  The article
[AFR] surveys most of the properties of surface groups and shows how they are the primary
motivating examples for much of combinatorial group theory.

A {\bf cyclically pinched one-relator group} is a
one-relator group of the following form
$$ G = \langle a_1,...,a_p,a_{p+1},...,a_n ; U = V \rangle$$
where $1 \ne U = U(a_1,...,a_p)$ is a cyclically reduced word in the free group $F_1$ on $a_1,...,a_p$ and
$1 \ne V = V(a_{p+1},...,a_n)$ is a cyclically reduced in the
free group $F_2$ on $a_{p+1},...,a_n$.

Clearly such a group is the free product of the free groups on $a_1,...,a_p$
and $a_{p+1},...,a_n$, respectively, amalgamated over the cyclic subgroups
generated by $U$ and $V$.
From the standard one-relator presentation for an orientable surface group of genus
$g \ge 2$ it follows that they are cyclically pinched
one-relator groups.  There is a similar decompositon in the nonorientable case. 

The HNN analogs of cyclically pinched one-relator groups are called {\bf conjugacy pinched
one-relator groups} and are also motivated by the structure of orientable surface groups.
In particular suppose
  $$ S_g = \langle a_1,b_1,...,a_g,b_g; [a_1,b_1]...[a_g,b_g] = 1\rangle.$$ 
If $b_g = t$ then $S_g$ is an HNN group of the form $$ S_g = \langle a_1,b_1,...,a_g,t; tUt^{-1} = V \rangle$$
where $U= a_g$ and $V = [a_1,b_1]...[a_{g-1},b_{g-1}]a_g$. Generalizing this we say that a {\bf
conjugacy pinched one-relator group} is a one-relator group of the  form
    $$ G = \langle a_1,...,a_n,t ; tUt^{-1} = V \rangle$$
where $1 \ne U = U(a_1,...,a_n)$ and $1 \ne V = V(a_{1},...,a_n)$ are cyclically reduced
 in the free group $F$ on $a_1,...,a_n$. 

The one-relator presentation for a surface group allows for a decomposition as a
cyclically pinched one-relator group in both the orientable and non-orientable cases and as a
conjugacy pinched one relator group in the orientable case (see [FRS]). In general cyclically pinched one-relator groups and conjugacy pinched one-relator groups have been shown to be extremely similar to surface groups. We refer to [FR] or [FRS], where a great deal more information about both cyclically pinched one-relator groups and conjugacy pinched one-relator groups is available.

In [FKMRR] several results were proved about the surface group conjectures. Recall that a  group $G$ is {\it fully
residually free} if given finitely many nontrivial elements $g_1,...,g_n$ in $G$ there is a
homomorphism $\phi:G \rightarrow F$, where $F$ is a free group, such that $\phi(g_i) \ne 1$ for all
$i=1,...,n$.  Fully residually free groups have played a crucial role in the study of equations and
first order formulas over free groups and in particular the solution of the Tarski problem (see
[KhM 1-5] and [Se 1-2]). Finitely generated fully residually free groups are also known as {\bf limit
groups}.  In this guise they were studied by Sela (see [Se 1-2] and [BeF 2]) in terms of studying
homomorphisms of general groups into free groups.
   
 
In  [FKMRR] several results concerning the Surface Group Conjectures
 were obtained:

\begin{theorem}[Thm 3.1, \protect{[FKMRR]}]  Suppose that $G$ is a finitely generated fully
residually free group with property IF.  Then G is either a free group or  a cyclically pinched one
relator group or a conjugacy pinched one relator group.
\end{theorem}

\begin{corollary}[Cor 3.2, \protect{[FKMRR]}] Suppose that $G$ is a finitely generated  fully residually free group with
property IF.  Then G is either free or every subgroup of finite index is freely indecomposable and
hence a one-relator group.
\end{corollary}

Furthermore we have:

\begin{theorem}[Thm 3.2, \protect{[FKMRR]}] Let $G$ be a finitely generated fully residually free group with property IF. 
Then $G$ is either hyperbolic or free abelian of rank 2.
\end{theorem}

These results depend on the fact that the fully residually free groups have JSJ-decompositions. Roughly a JSJ-decomposition
of a group $G$ is a splitting of $G$ as a graph of groups with abelian edges which is canonical in
that it encodes all other such abelian splittings.  If each edge is cyclic it is called a {\it cyclic JSJ-decomposition} (see [FKMRR]).

Being fully residually free provides a graph of groups decomposition. Then property IF will imply the finite index property, as follows:  

\begin{theorem}[Thm 3.3, \protect{[FKMRR]}]\label{33}  Let $G$ be a nonfree cyclically pinched or conjugacy pinched one-relator
group with property IF.  Then each subgroup of finite index is again a cyclically pinched or
conjugacy pinched one-relator group. 
\end{theorem}

The proof of Theorem \ref{33} used the subgroup theorems for free products with amalgamation and HNN groups as
described by Karrass and Solitar [KS 1,2].

Our proof of the Surface Group Conjecture C in Theorem \ref{mthm2} combines the results in [FKMRR] with the following statement, which is a rewording of a recent result of H.Wilton [W].

\begin{theorem}[see Cor 4, \protect{[W]}] Let $G$ be a hyperbolic one-ended cyclically pinched one-relator group or 
a hyperbolic one-ended conjugacy pinched one-relator group.
Then either $G$ is a surface group, or $G$ has a finitely generated non-free 
subgroup of infinite index.
\end{theorem} 
 
\section {Main Results}

Our main result shows that Surface Group Conjecture B is true if $G$ is a cyclically pinched or conjugacy pinched one-relator.

\begin{theorem}\label{mthm} (1) Let $G$ be a cyclically pinched one-relator group 
satisfying property IF.  Then $G$ is a free group or a surface group.

$\hphantom{xx}$ (2) Let $G$ be a conjugacy pinched one-relator group 
satisfying property IF.  Then $G$ is a free group, a surface group or a solvable Baumslag-Solitar group.
\end{theorem}

Before proving Theorem \ref{mthm}, we settle the Surface Group Conjecture C in the following theorem. This result appears in [Cor. 5, W] for one-ended limit groups.

\begin{theorem}\label{mthm2} Suppose that $G$ is a finitely generated nonfree freely
indecomposable fully residually free group with property IF. Then G is a surface group. That is, Surface Group Conjecture C is true.
\end{theorem}

\begin{proof}  Suppose that $G$ is a finitely generated freely indecomposable fully residually free group with property IF. If $G$ is free abelian we are done since the free abelian group of rank $2$ is a surface group, and of higher rank does not posses IF. If it is not free abelian then from Theorem 1.2 from [FKMRR] it follows that $G$ is hyperbolic. 

Since $G$ is assumed to be a finitely generated freely indecomposable fully residually group and satisfies Property IF, from Theorem 2.1 in [FKMRR] it follows that $G$ must be either a cyclically pinched one relator group or a conjugacy pinched one relator group.

To apply Wilton's Theorem we show  that it must be one-ended. Let $e(H)$ denote the number of ends of a group $H$. From a theorem of Stallings [St]
$e(H) = 0$ if and only if $H$ is finite and 
$e(H) > 1$ if and only if $H$ has a non-trivial decomposition as a free product with 
amalgamation
with finite amalgamated subgroup or a non-trivial decomposition as a 
HNN-group with
finite associated subgroup.

Let $G$ be the group as in the statement of the theorem. Certainly $G$ is not finite so $e(G) > 0$. Since $G$ satisfies Property IF it follows that if $e(G) > 1$ then the amalgamated subgroup (or the associated subgroup) in $G$ 
has to be trivial. That implies that if our $G$ is not one-ended then there is a free 
infinite cyclic factor and $G$ is free or has a non-free subgroup of infinite index. Since $G$ is freely indecomposable it follows that it must be one-ended and Wilton's result applies.

Therefore $G$ is either a surface group or has a nonfree subgroup of infinite index. Again from property IF $G$ must be a surface group settling Surface Group Conjecture C.

\end{proof}  

We now give the proof for cyclically pinched and conjugacy pinched one-relator groups.

\begin{proof} (Theorem 3.1) (1)  Let $G$ be a cyclically pinched one-relator group amalgamated via $U = V$ and suppose that $G$ is nonfree. 
Suppose that not both $U$ and $V$ are proper powers. Then by results of Juhasz and Rosenberger [JR],  Bestvina and Feighn [BeF 1] and
Kharlampovich and Myasnikov [KhM], $G$ must be hyperbolic. Since $G$ satisfies Property IF, as in the proof of Theorem 3.2, $G$ must be one-ended and hence Wilton's theorem applies to give that $G$ must be a surface group. 

Suppose now that both $U$ and $V$ are proper powers. Let $U = g^n,n > 1$ and $V = h^m,m >1$.
 If $G = \langle g,h : g^2 = h^2 \rangle$ then $G$ is a nonhyperbolic, nonorientable surface group of genus 2. 

Now assume that $G$ is not isomorphic to a group $\langle a,b; a^2 = b^2 \rangle$.  Then consider the subgroup $H = \langle g^n,gh \rangle$.  $H$ is free abelian of rank 2 and further $H$ has infinite index in $G$. To see this introduce the relations $g^n = h^m = 1$.  Then the image of $G$ is a nontrivial free product, not isomorphic to the infinite dihedral group, and the image of $H$ is infinite cyclic. However $G$ is assumed to have Property IF and hence this case is impossible.

Thus all cyclically pinched one-relator groups with Property IF must be either free or a surface groups.

(2) Now let $G$ be a conjugacy pinched one-relator group satisfying Property IF and assume $G$ is nonfree. Suppose first that $U$ and $V$ are not both proper powers. Assume that $U$ and $V$ are conjugately separated in $F$, where $F$ is the group generated by $a_1,...,a_n$. This means that $\langle U \rangle \cap x\langle V \rangle x^{-1}$ is finite for all $x \in F$.  Since $G$ is torsion-free this intersection must be trivial.  By a result of Kharlampovich and Myasnikov [KhM4], $G$ is then hyperbolic. Since $G$ satisfies Property IF then, as in the proof of Theorem 3.2,  $G$ is one-ended and hence Wilton's theorem applies to give that $G$ is a surface group.

Thus  $\langle U \rangle \cap x \langle V \rangle x^{-1}$ is infinite, in fact infinite cyclic, for some $x \in F$. After a suitable conjugation and a possible interchange of $U$ and $V$ we may assume that $U$ is not a proper power in $F$ and $V = U^k$ for some $k \ne 0$.  Let $K$ be the subgroup generated by $U$ and $t$. By normal form arguments $K$ has a presentation $K = \langle U,t : tUt^{-1} = U^k \rangle$ (see [FRS]). If $F$ is free of rank $> 1$, that is if $n > 1$, then $K$ is a nonfree subgroup of infinite index in $G$.  This can be seen as follows. In $G$ introduce the relations $t = U = 1$. Then the image of $G$ is a one-relator group $\langle a_1,...,a_n ; U =1 \rangle$ which is infinite by the Freiheitssatz. Therefore this case cannot occur since $G$ satisfies Property IF.

Now let $F$ be cyclic and let $a_1 =  a$, so $G = \langle a,t; tat^{-1} =  a^p\rangle$ with $p = \pm k$.  Then $G$ is a solvable Baumslag-Solitar group. 

Finally let both $U$ and $V$ be proper powers so suppose that $U = g^n, n > 1$ and $V = h^m,m>1$. By normal form arguments the subgroup $N$ generated by $w = tgt^{-1}$ and $h$ has a presentation $N = \langle w,h; w^n = h^m \rangle$  ( see [FRS]). Note that the subgroup $H$ generated by $w^n$ and $wh$ is free abelian of rank 2 and has infinite index in $G$.  Since $G$ is assumed to have Property IF this does not occur. 

Altogether,  if $G$ is a conjugacy pinched one-relator group with Property IF then $G$  must be either free, a surface group or a solvable Baumslag-Solitar group.  
\end{proof}

\section{Some Observations on the General Conjecture}

Here we make some straightforward observations based on the proofs of Theorems 3.1 and 3.2 that might have a bearing on the general Surface Group Conjecture.  From the proofs in [FKMRR] and the proofs of Theorems 3.1 and 3.2 we have the following.

\begin{lemma} Let $G$ be nonfree and  have a graph of groups decomposition with Property IF. Then each factor must be free and $G$ is one-ended.
\end{lemma}

\begin{lemma} Let $G$ be nonfree and have a graph of groups decomposition with cyclic edge groups and with Property IF. Then $G$ is a cyclically pinched or conjugacy pinched one-relator group, $G$ is one-ended and hence either a surface group or a solvable Baumslag-Solitar group.

\end{lemma}

We would like to note how strong Property IF is.  A rewording of Lemma 4.2 implies that if $G$ is not a free, not a surface, and not a Baumslag-Solitar group, then it cannot have a graph of groups decomposition with cyclic edge groups.

Now we consider noncyclic one-relator groups. Certainly if $G$ has property IF it must be torsion-free. Let $G$ be a torsion-free one-relator group with property IF. Using the standard Magnus breakdown $G$ can be taken to be an HNN group. Hence by Property IF the base $F$ is a free group and as in the proofs in section 3 it must be one-ended.  Altogether then if $G$ is a torsion-free finitely generated one-relator group with property IF then $G$ has a presentation as an HNN extension with base group $F$, a free group.  Then by rewriting we can get a presentation of the form
$$ G = <F,t ; tU_1t^{-1} = U_2,tU_2t^{-1} = U_3,...,tU_{k-1}t^{-1} = U_k>$$
for some free group words $U_1,U_2,...,U_k$.  If $k =2 $ then it is a conjugacy pinched one-relator group and hence $G$ is either free, a surface group or a solvable Baumslag-Solitar group. The general surface group conjecture would then be true if the following one is true.

{\bf Conjecture}. \textit{Suppose that
$$ G = <F,t ; tU_1t^{-1} = U_2,tU_2t^{-1} = U_3,...,tU_{k-1}t^{-1} = U_k>$$
with $F$ a finitely generated nonabelian free group  and $U_1,U_2,...,U_k$  some nontrivial free group words. If $k > 2$ there must be a finite index subgroup that is not a one-relator group. }
 
\section*{Acknowledgments}

The authors were partially supported by the Marie Curie Reintegration Grant 230889.

 \section*{References}

\noindent \lbrack  AFR] P.Ackermann, B.Fine and G.Rosenberger,  On Surface Groups: Motivating
Examples in Combinatorial Group Theory,\textbf{ Groups St. Andrews 2005}. Cambridge University Press, 2007, 96--129

\vspace{.3cm}

\noindent \lbrack BeF 1] M. Bestvina and M. Feighn,  A combination theorem for negatively curved
groups, \textbf{ J. Diff. Geom.}, 35,  1992,  85-101

\vspace{.3cm}

\noindent \lbrack BeF 2] M. Bestvina and M. Feighn,  Notes on Sela's Work: Limit
Groups and Makanin-Razborov Diagrams, preprint

\vspace{.3cm}

\noindent \lbrack  CoLe] D.Collins and F.Levin, Automorphisms and Hopficity of Certain Baumslag-Solitar Groups,  \textbf{Archiv. Math.}, 40, 1983, 385--400

\vspace{.3cm}

\noindent \lbrack FKMRR] B. Fine, O.Kharlampovich, A.Myasnikov,V.Remeslennikov and G. Rosenberger, On the Surface Group Conjecture, \textbf{Scienta: Math Series A }, 1, 2008, 1-15

\vspace{.3cm}

\noindent \lbrack FR] B. Fine and G. Rosenberger, \textbf{ Algebraic Generalizations of Discrete Groups},
 Marcel-Dekker,  1999 
 
\vspace{.3cm}

\noindent \lbrack FRS] B.Fine, G.Rosenberger and M. Stille, Conjugacy Pinched and
Cyclically Pinched One-Relator Groups, \textbf{ Revista Math. Madrid }, 10,  1997,
 207--227
 
\vspace{.3cm}

\noindent \lbrack GS] A. Gaglione and D. Spellman,  Some Model Theory of Free Groups and Free
Algebras, \textbf{ Houston J. Math }, 19,  1993,  327-356

\vspace{.3cm}

\noindent \lbrack GKM] D.Gildenhuys, O. Kharlampovich and A. Myasnikov,  CSA Groups and Separated
Free Constructions, \textbf{Trans. Amer. Math. Soc.}, 350,  1998,  571-613

\vspace{.3cm}

\noindent \lbrack HKS 1] A. Hoare, A. Karrass and D. Solitar,  Subgroups of finite
index of Fuchsian groups, \textbf{ Math. Z. }, 120,  1971,  289--298

\vspace{.3cm}
 
\noindent \lbrack HKS 2] A. Hoare, A. Karrass and D. Solitar, Subgroups of
infinite index in Fuchsian groups, \textbf{ Math. Z. }, 125,  1972, 
59--69 

\vspace{.3cm}

\noindent \lbrack JR] A. Juhasz and G. Rosenberger,  On the Combinatorial Curvature of Groups of 
F-type and Other One-Relator Products of Cyclics, \textbf{Cont. Math.}, 169,  1994,  373-384

\vspace{.3cm}

\noindent \lbrack KS1] A. Karrass and D.Solitar,  The Subgroups of a Free Product of Two Groups with an Amalgamated Subgroup, \textbf{Trans. Amer. Math. Soc.}, {\bf Vol 150}, (1970), 227-255

\vspace{.3cm}

\noindent \lbrack KS2] A. Karrass and D.Solitar,   Subgroups of HNN Groups and Groups with One Defining Relation, \textbf{Can. J. Math. Soc.}, {\bf Vol 23}, (1971), 627-643

\vspace{.3cm}

\noindent \lbrack KhM] O. Kharlampovich and A. Myasnikov,  Hyperbolic groups and free constructions;
\textbf{Trans. Amer. Math. Soc.} 350 (1998), 571-613.

\vspace{.3cm}

\noindent \lbrack  KhM 1] O. Kharlamapovich and A.Myasnikov, 
Irreducible affine varieties over a free group: I. Irreducibility
of quadratic equations and Nullstellensatz, \textbf{  J. of Algebra}, 200, 1998, 472-516

\vspace{.3cm}

\noindent \lbrack KhM 2] O. Kharlamapovich and A.Myasnikov, 
Irreducible affine varieties over a free group: II. Systems in
triangular quasi-quadratic form and a description of residually
free groups, \textbf{ J. of Algebra}, 200, 1998, 517-569

\vspace{.3cm}

\noindent \lbrack KhM 3] O. Kharlamapovich and A.Myasnikov, 
Description of fully residually free groups and Irreducible affine
varieties over free groups, \textbf{ Summer school in Group Theory in
Banff, 1996, CRM Proceedings and Lecture notes }, 17,  1999,
 71-81
 
\vspace{.3cm}

\noindent \lbrack KhM 4] O.Kharlamapovich and A.Myasnikov, 
Hyperbolic Groups and Free Constructions, \textbf{ Trans. Amer. Math.
Soc. } 350, 2,  1998,  571-613

\vspace{.3cm}

%

\noindent \lbrack KhM 5] O. Kharlamapovich and A.Myasnikov, 
Description of fully residually free groups and irreducible affine
varieties over a free group,  \textbf{  CRM Proceeding and Lecture
Notes: Summer School in Group Theory in Banff 1996}, 17,
 1998,  71-80 
  
\vspace{.3cm}

\noindent \lbrack Ko] Y.I. Merzlyakov, \textbf{Kourovka Notebook - Unsolved Problems in Group Theory}

\vspace{.3cm}

\noindent \lbrack Re ] V.N. Remeslennikov,   $\exists$-free groups, \textbf{ Siberian Mat. J.}, 30,
 1989,  998--1001
 
\vspace{.3cm}

\noindent \lbrack St ] J. Stallings, On torsion-free groups with infinitely many ends, \textbf{ Annals 
of Math.} (2) 88 (1968), 312-334  

 \vspace{.3cm} 

\noindent \lbrack Se 1] Z. Sela,  Diophantine Geometry over Groups I: Makanin-Razborov Diagrams, 
\textbf{ Publ. Math. de IHES}, 93,  2001,  31-105 

\vspace{.3cm}

%

%
%
%
%
\noindent \lbrack Se2] Z. Sela,  Diophantine Geometry over Groups V: Quantifier Elimination,
\textbf{Israel Jour. of Math.}, 150,  2005,  1-97 

\vspace{.3cm}
%
%

\noindent \lbrack W] H. Wilton, One ended subgroups of graphs of free groups with cyclic 
edge groups \textbf{Geom. Topol. } (to appear) 

\bigskip
\noindent Laura Ciobanu, University of Fribourg, Department of Mathematics, Chemin du Mus\'ee 23, 1700 Fribourg, Switzerland.\\
 \textit{E-mail}: \textrm{laura.ciobanu@unifr.ch}\\
 
 \noindent Benjamin Fine, Fairfield University, Fairfield, CT 06430, USA\\
 \textit{E-mail}: \textrm{fine@fairfield.edu}\\
 
 \noindent Gerhard Rosenberger,Fachebereich Mathematik, University of Hamburg,Bundestrasse 55, 20146 Hamburg, Germany\\
 \textit{E-mail}: \textrm{Gerhard.Rosenberger@math.uni-hamburg.de}

\end{document}